\documentclass[11pt]{amsart}
\usepackage{amssymb,latexsym}
\usepackage{graphicx}
\usepackage{amsmath}
\usepackage{wasysym}
\usepackage{enumerate}
\usepackage{color}
\usepackage[super]{nth}
\usepackage{hyperref}
\usepackage{tikz}
\usetikzlibrary{arrows.meta}
\usepackage{pgfplots}
\usepackage{amsthm}

\usepackage{anyfontsize}

\newcommand{\bb}[1]{\mathbb{#1}} 
\newcommand{\cali}[1]{\mathcal{#1}} 

\newcommand{\abs}[1]{\ensuremath{\left| #1 \right|}} 
\newcommand{\set}[1]{\ensuremath{\lbrace #1 \rbrace}} 

\newtheorem{thm}{Theorem}[section]
\newtheorem{lemma}{Lemma}[section]
\newtheorem{defn}[lemma]{Definition}
\newtheorem{prop}[lemma]{Proposition}

\definecolor{purple}{RGB}{200, 0, 255}
\definecolor{orange}{RGB}{255,140, 0}
\tikzstyle{flowNode} = [rectangle, rounded corners, minimum width = 3cm, minimum height = 1cm, text width = 3cm, text centered, draw = black]

\newcommand{\propL}{end-normal }

\begin{document}
\title[FFTP Implies Autostackability]{The Falsification by Fellow Traveler Property Implies Geodesic Autostackability} 
\author[A.~DeClerk]{Ash DeClerk}
\address{Department of Mathematics\\
        University of Nebraska\\
         Lincoln NE 68588-0130, USA}
\email{declerk.gary@huskers.unl.edu}
\thanks{2020 {\em Mathematics Subject Classification}. 
  20F65; 20F10, 68Q42}
\begin{abstract}
Groups with the falsification by fellow traveler property are known to have solvable word problem, but they are not known to be automatic or to have finite convergent rewriting systems. In this paper, we show that these groups admit a generalization of the two properties; namely, they are geodesically autostackable. As a key part of proving this, we show that a wider class of groups, namely groups with a weight non-increasing synchronously regular convergent prefix-rewriting system, have a bounded regular convergent prefix-rewriting system.
\end{abstract}

\maketitle
\renewcommand*{\thethm}{\Alph{thm}}

\section{Introduction} \label{intro}
For a group $G$ with a finite inverse-closed set $A$ which generates $G$, the falsification by fellow traveler property (FFTP) is a purely geometric property of the Cayley graph which was introduced by Neumann and Shapiro \cite{NS1995}.

\begin{defn}
\cite{NS1995} The pair $(G, A)$ has the \emph{falsification by fellow traveler property (FFTP)} if there exists a constant $k$ such that for each non-geodesic word $u$ in the generators and their inverses, there exists a word $v$ such that $\abs{v} < \abs{u}$, $u =_G v$, and $u$ and $v$ $k$-fellow travel. That is, $d(u(n), v(n)) \leq k$ for all $n \in \bb{N}$. Such a word $v$ is called a \emph{witness} of $u$.
\end{defn}

Cannon gave an example showing that FFTP depends on the generating set for a group \cite{NS1995}, but many choices of group $G$ have at least one generating set $A$ such that the pair $(G, A)$ has FFTP. Some examples include virtually abelian groups and geometrically finite hyperbolic groups \cite{NS1995}, Garside groups \cite{Holt2010}, Artin groups of large type \cite{HoltRees2012}, Coxeter groups \cite{Noskov2001}, and groups acting cellularly on locally finite CAT(0) cube complexes with a simply transitive action on the vertices \cite{Noskov2000}. Certain particularly nice groups are known to have FFTP for every choice of generating set; abelian groups \cite{NS1995} and finite groups are two examples, though even virtually abelian groups can have generating sets such that $(G, A)$ does not have FFTP \cite{NS1995} --- we will use one such example in Section \ref{noAConverse}. Furthermore, pairs $(G, A)$ with FFTP have particularly nice properties, including regular geodesic language \cite{NS1995}, at most quadratic isoperimetric inequality \cite{Elder2005}, and almost convexity \cite{Elder2005}.

Autostackability was introduced by Brittenham, Hermiller, and Holt \cite{BHH2014} as a generalization of automaticity and the property of having a finite convergent rewriting system, and is defined in terms of the Cayley graph $\Gamma_{G, A}$ for a group $G$ with inverse-closed generating set $A$. Given a spanning tree $T$ in $\Gamma_{G, A}$, a \emph{regular bounded flow function} for the triple $(G, A, T)$ is a function $\Phi$ from the set of directed edges of $\Gamma_{G, A}$ to the set of directed paths in $\Gamma_{G, A}$ such that:
\begin{itemize}
\item (same endpoints) for any directed edge $e$ of $\Gamma_{G, A}$, the path $\Phi(e)$ has the same initial vertex and the same terminal vertex as $e$;
\item (boundedness) there exists a constant $k \geq 0$ such that for each edge $e$ of $\Gamma_{G, A}$, the path $\Phi(e)$ has length at most $k$;
\item (fixed tree edges) if the undirected edge underlying $e$ lies in the tree $T$, then $\Phi(e) = e$;
\item (termination) there is no infinite sequence of edges $e_1, e_2, e_3, \dots$ such that each $e_{i + 1}$ lies in the path $\Phi(e_i)$ and each $e_i$ lies outside the spanning tree $T$; and
\item (regularity) the language of triples $(\text{nf}_{\Phi}(\iota(e)), \text{label}(e), \text{label}(\Phi(e)))$ is a synchronously regular language (where $\text{nf}_{\Phi}(v)$ is the label of the unique non-backtracking path from $1_G$ to $v$ in $T$, $\iota(e)$ is the starting vertex of the edge $e$, and $\text{label}(p)$ is the label of the path $p$).
\end{itemize} (See Section \ref{notation} for notation and background on regular languages.)

\begin{defn}
\cite{BHH2014} A group is \emph{autostackable} if it admits a regular bounded flow function.
\end{defn}

Many classes of groups have been shown to be autostackable for certain generating sets, including all prefix-closed automatic groups \cite{BHH2014} and all groups with a finite convergent rewriting system \cite{BHH2014}, Thompson's group $F$ \cite{CGHJS2020}, closed 3-manifold groups \cite{BHS2018}, the Baumslag-Gersten group \cite{HMP2018}, and Stallings' finitely presented group which is not of type $FP_3$ \cite{BHJ2016}. 

Throughout this paper, we will also consider convergent prefix-rewriting systems:

\begin{defn} \label{CPRS} \cite{BHH2014} A \emph{convergent prefix-rewriting system (CP-RS)} for a group $G$ is a pair consisting of a finite alphabet $A$ and a set $R \subseteq A^\ast \times A^\ast$ of ordered pairs of words over $A$ such that $G$ is presented as a monoid by $$G = Mon\langle A \mid \set{u = v \mid (u, v) \in R} \rangle$$ and the set of rewritings $uw \to_R vw$ with $(u, v) \in R$ and $w \in A^\ast$ satisfies:
\begin{itemize}
\item (termination) there is no infinite chain of rewritings $$w_1 \to_R w_2 \to_R w_3 \to_R \cdots$$ and
\item (normal forms) each element of $G$ is represented by exactly one irreducible word (i.e. a word which cannot be rewritten) over $A$.
\end{itemize}
\end{defn}

A CP-RS $R$ is said to be \emph{length non-increasing} if for each $(u, v) \in R$, we have $\abs{v} \leq \abs{u}$ and \emph{bounded} if there exists a constant $k$ such that for each pair $(l, r)$ in $R$, we have $l = pl'$, $r = pr'$, and $\abs{l'} + \abs{r'} \leq k$. We will say that a CP-RS is \emph{\propL} if for all pairs $(u, v) \in R$ with $u = wl$ for some irreducible word $w \in A^\ast$ and some letter $l \in A$, $\text{last}(v) = \text{last}(\text{nf}_R(v))$. A CP-RS is \emph{(synchronously) regular} if the set $R$ is a (synchronously) regular language. If we have a system of positive finite weights on $A$ (that is, a function $wt: A \to \mathbb{R}^+$), then we define $wt(a_1a_2 \cdots a_n) = \sum_{i = 1}^n wt(a_i)$; then $R$ is \emph{weight non-increasing} if for each $(u, v) \in R$ we have $wt(v) \leq wt(u)$. In particular, a length non-increasing CP-RS is also weight non-increasing, where $wt(a) = 1$ for all $a \in A$.

Bounded synchronously regular convergent prefix-rewriting systems and autostackable groups are intimately connected. In particular:

\begin{prop}\label{brcprs}\cite{BHH2014} A group $G$ is autostackable if and only if $G$ admits a bounded synchronously regular convergent prefix-rewriting system. Moreover, there is an algorithm which, given a bounded synchronously regular convergent prefix-rewriting system $R$, can construct an autostackable structure such that the normal forms of this autostackable structure are the same as the normal forms of $R$. \end{prop}

We extend this result to groups with weight non-increasing synchronously regular convergent prefix-rewriting systems with the following theorem:

\begin{thm} \label{lniToAutost} Suppose $G = \langle A \rangle$ has a weight non-increasing synchronously regular CP-RS $R$. Then $G$ has an autostackable structure which has the same normal forms as $R$. \end{thm}

We will also discuss a more restrictive version of autostackability, namely geodesic autostackability:

\begin{defn}\label{geoAutost} \cite{BH2015} Define $\alpha: E(\Gamma_{G, A}) \to \bb{Q}$ by $$\alpha(e) = \frac{1}{2}\left( d_{G, A}(1_G, \iota(e)) + d_{G, A}(1_G, \tau(e))\right).$$
A group $G$ is \emph{geodesically autostackable} if $G$ has an autostackable structure with normal form set $\cali{N}$ and flow function $\Phi$ such that each element of $\cali{N}$ labels a geodesic in $\Gamma_{G, A}$, and whenever $e', e \in E(\Gamma_{G, A})$ with $\Phi(e') \neq e'$ and $e'$ an edge in the path $\Phi(e)$, we have $\alpha(e') < \alpha(e)$.
\end{defn}

Geodesically autostackable structures are significantly more restrictive than autostackable structures. Geodesically autostackable structures provide a regular set of geodesic normal forms and require that the pair $(G, A)$ is almost convex:

\begin{prop}\cite{BH2015} A geodesically autostackable group $G = \langle A \rangle$ is almost convex with respect to the generating set for which $G$ has a geodesically autostackable structure. \end{prop}

(A definition of almost convexity is provided as Definition \ref{almCon}.)

We strengthen this result with Theorem \ref{lniToAlmCon} from the present paper:

\begin{thm} \label{lniToAlmCon} Suppose $G = \langle A \rangle$ has a length non-increasing synchronously regular convergent prefix-rewriting system. Then $G$ is almost convex. \end{thm}

We also give sufficient conditions for our proof of Theorem \ref{lniToAutost} to build a geodesically autostackable structure from a length non-increasing synchronously regular CP-RS with the following theorem:

\begin{thm} \label{extraToGeo} Suppose that $G = \langle A \rangle$ admits a length non-increasing \propL synchronously regular CP-RS $R$. Then the autostackable structure for $G$ constructed in the proof of Theorem \ref{lniToAutost} is a geodesically autostackable structure. \end{thm}

Our main theorem is that all pairs $(G, A)$ with FFTP are geodesically autostackable, which we prove by applying Theorem \ref{extraToGeo}:

\begin{thm} \label{mainThm} [Main Theorem] Suppose the pair $(G, A)$ has the falsification by fellow traveler property. Then
\begin{enumerate}[(a)]
\item $G$ admits a length non-increasing \propL synchronously regular convergent prefix-rewriting system $R$; and
\item $G$ is geodesically autostackable.
\end{enumerate} \end{thm}

There are several related implications between the theorems proved in this paper and other theorems mentioned in this section and in section \ref{notation}. In order to better visualize these implications, we have included a flowchart with the results of this paper and similar results from previous papers as Figure \ref{fig: flowchart}.

The current paper is organized as follows: In Section \ref{notation}, we provide notation, definitions, and theorems which are used in the remainder of the paper. In Section \ref{ThmBC}, we give a constructive proof of Theorem \ref{lniToAutost}, prove Theorem \ref{extraToGeo} as a consequence of adding additional hypotheses to Theorem \ref{lniToAutost}, and prove Theorem \ref{lniToAlmCon}. In Section \ref{ThmA}, we prove Theorem \ref{mainThm}, the main theorem of this paper. In Section \ref{noAConverse}, we show that there are examples of pairs $(G, A)$ with length non-increasing synchronously regular CP-RS which do not have FFTP with the associated generating set. As a consequence of this, the converse of part (a) of the main theorem fails.

\begin{figure}
\begin{tikzpicture}[scale = 2]
\node (FFTP) [flowNode] at (0, 0) {FFTP};
\node (niceCPRS) [flowNode] at (4, 0) {Length-non-increasing \propL synchronously regular CP-RS};
\node (geoAS) [flowNode] at (2, -1.5) {Geodesically Autostackable};
\node (AC) [flowNode] at (0, -3) {Almost Convex};
\node (LNICPRS) [flowNode] at (4, -3) {Length-non-increasing CP-RS};
\node (AS) [flowNode] at (1, -4.5) {Autostackable};
\node (bcprs) [flowNode] at (4, -4.5) {Bounded synchronously regular CP-RS};
\draw [double distance = 5pt, -{Stealth[length = 6mm, width = 4mm]}] (FFTP) -- node[above] {Thm \ref{mainThm} part (a)} (niceCPRS);
\draw [double distance = 5pt, -{Stealth[length = 6mm, width = 4mm]}] (FFTP) -- node[above right] {Thm \ref{mainThm} part (b)} (1.5, -1.125) -- (geoAS);
\draw [double distance = 5pt, -{Stealth[length = 6mm, width = 4mm]}] (FFTP) -- node[left] {\cite{Elder2005}} (AC);
\draw [double distance = 5pt, -{Stealth[length = 6mm, width = 4mm]}] (niceCPRS) -- node[above left] {Thm \ref{extraToGeo}} (geoAS);
\draw [double distance = 5pt, -{Stealth[length = 6mm, width = 4mm]}] (geoAS) -- node[above left] {\cite{BH2015}} (AC);
\draw [double distance = 5pt, -{Stealth[length = 6mm, width = 4mm]}] (niceCPRS) -- node[right] {Immediate} (LNICPRS);
\draw [double distance = 5pt, -{Stealth[length = 6mm, width = 4mm]}] (LNICPRS) -- node[above] {Thm \ref{lniToAlmCon}} (AC);
\draw [double distance = 5pt, -{Stealth[length = 6mm, width = 4mm]}] (geoAS) -- (1.5, -3) -- node[below right] {Immediate} (AS);
\draw [double distance = 5pt, -{Stealth[length = 6mm, width = 4mm]}] (LNICPRS) -- node[right] {Thm \ref{lniToAutost}} (bcprs);
\draw [double distance = 5pt, {Stealth[length = 6mm, width = 4mm]}-{Stealth[length = 6mm, width = 4mm]}] (AS) -- node[below] {\cite{BHH2014}} (bcprs);
\end{tikzpicture}
\caption{A flowchart of related results from this and other papers}
\label{fig: flowchart}
\end{figure}
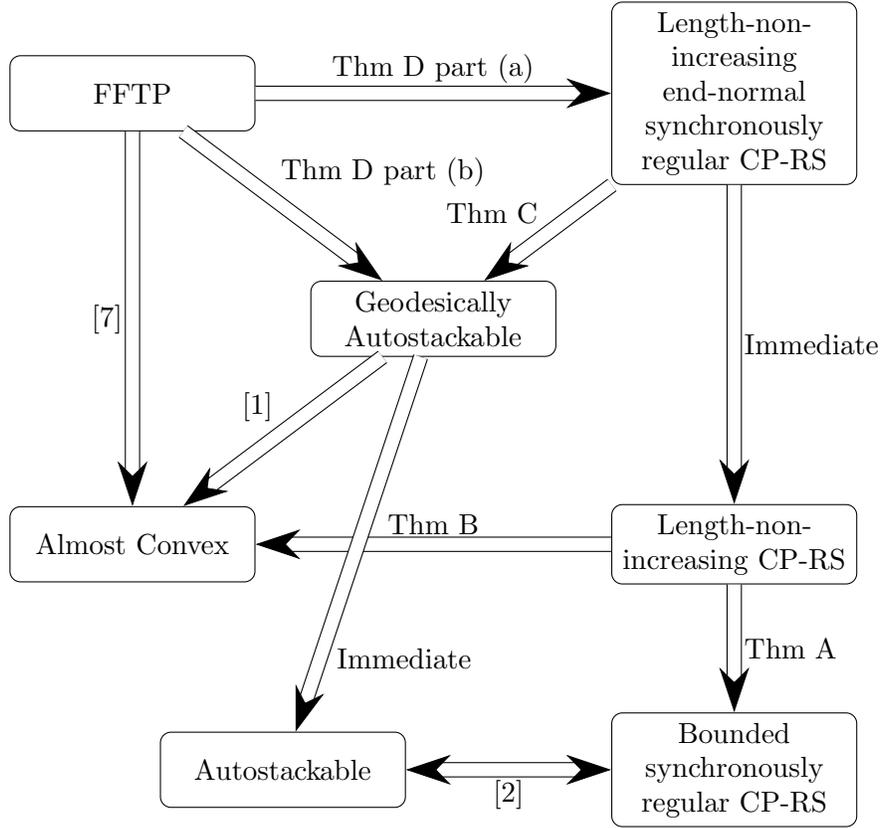

\section{Notation, Definitions, and Background} \label{notation}

Throughout this paper, let $G$ be a group with finite generating set $A$ which is closed under inversion.

Let $\Gamma_{G, A}$ be the Cayley graph of $G$ with respect to generating set $A$. Then for a directed edge $e$ of $\Gamma_{G, A}$, we use $\iota(e)$ to denote the initial vertex of $e$, $\tau(e)$ to denote the terminal vertex of $e$, and $\text{label}(e)$ to denote the label from $A$ associated to $e$. We also denote the vertex corresponding to the identity element of $G$ by $1_G$, and $e_{y, a}$ to denote the directed edge with initial vertex $y$ and label $a$. For a path $p$ in $\Gamma_{G, A}$, we use $\text{label}(p)$ to denote the word in $A^\ast$ obtained by concatenating the labels of each edge in $p$.

For words $u, v \in A^\ast$, we use $u =_{A^\ast} v$ to mean that $u$ and $v$ are identical words, and $u =_G v$ to mean that $u$ and $v$ represent the same group element in $G$. Given a CP-RS $R$, we say that $u =_R v$ if there exist words $u_1, \dots, u_n, v_1, \dots v_m, w \in A^\ast$ with $u \to_R u_1 \to_R \cdots \to_R u_n \to_R w$ and $v \to_R v_1 \to_R \cdots \to_R v_m \to_R w$. Note that if $R$ is a CP-RS for $G$, then $u =_R v$ if and only if $u =_G v$. We use $\abs{u}$ to denote the length of $u$, $u(n)$ to denote the prefix of $u$ with exactly $n$ letters if $n \leq \abs{u}$ and all of $u$ if $n \geq \abs{u}$, $\text{first}(u)$ to denote the first letter of $u$, and $\text{last}(u)$ to denote the last letter of $u$. For a group element $g$ or a word $u$ representing a group element $g$, we use $\abs{g}_{G, A}$ or $\abs{u}_{G, A}$ to denote the minimum length of any word $w \in A^\ast$ with $w =_G g$ or $w =_G u$. When we have a set of normal forms $\cali{N}$ (i.e. a set of unique representatives of group elements), we will use $\text{nf}(u)$ to denote the normal form of the group element represented by $u$, and where we have multiple normal form sets we will use subscripts based on the source of the normal forms (e.g. $\text{nf}_R(u)$) to specify which normal forms we mean. We will call any word which is not a normal form \emph{reducible}, and any reducible word of the form $ul$ for some normal form $u \in A^\ast$ and letter $l \in A$ \emph{minimally reducible}. We use $\overline{u}$ to denote the group element represented by $u$, use $\lambda$ to denote the empty word, and use $d_{G, A}(\overline{u}, \overline{v})$ to denote the path metric on $\Gamma_{G, A}$.

\begin{defn}
For an alphabet $A$ with a strict total ordering $\prec$, a word $u$ is said to be \emph{short reverse lexicographically smaller} than a word $v$ ($u <_{srev} v$) if $u \neq v$ and either
\begin{itemize}
\item $\abs{u} < \abs{v}$; or
\item $\abs{u} = \abs{v}$ and for each $s \in A^\ast$ such that $u =_{A^\ast} u's$ and $v =_{A^\ast} v's$ for some $v', w' \in A^\ast$, we have $\text{last}(u') \preceq \text{last}(v')$.
\end{itemize}
\end{defn}

The short reverse lexicographic ordering given above is essentially the shortlex ordering reading right-to-left rather than left-to-right, and we use it in this paper rather than shortlex to make certain proofs and constructions simpler.

\subsection{Formal Language Theory}
In this subsection we provide definitions and background related to language theory. The reader can refer to \cite{HopUll1979} for a more detailed treatment. A \emph{language} over $A$ is a subset of $A^\ast$. The \emph{regular languages} consist of all finite languages, along with the closure under finitely many unions, intersections, complements, concatenations ($P \cdot S = \set{ps \mid p \in P, s \in S}$), and Kleene stars ($S^0 = \set{\lambda}$, $S^i = S^{i - 1} \cdot S$ for all natural numbers $i$, and $S^\ast = \bigcup_{i = 0}^\infty S^i$). While it is not immediate, the class of regular languages is also closed under quotients ($P / S = \set{w \in A^\ast \mid \text{there exists } s \in S \text{ such that }ws \in P}$) \cite[Theorem 3.6]{HopUll1979}. We use $A^{k}$ to denote the set of all words over $A$ of length exactly $k$, and $A^{\leq k}$ to denote the set of all words over $A$ with length at most $k$.

\begin{defn} \label{FSA} \cite[Page 17]{HopUll1979}
A \emph{finite state automaton (FSA)} consists of a finite set of states $Q$, a finite input alphabet $A$, a transition function $\delta: Q \times A \to Q$, an initial state $q_0$, and a set of accepting states $P \subseteq Q$. Given a finite state automaton $M$, $\widehat{\delta}: Q \times A^\ast \to Q$ is the function given by $\widehat{\delta}(q, \lambda) = q$ and $\widehat{\delta}(q, aw) = \widehat{\delta}(\delta(q, a), w)$. A word $u$ is \emph{accepted} by $M$ if $\widehat{\delta}(q_0, u) \in P$, and the \emph{language accepted by $M$} is the set of all words which are accepted by $M$. 
\end{defn}

We will at times view finite state automata as labeled directed graphs with vertex set $Q$. Under this view, the edges are given by the transition function, with $\delta(q, a) = q'$ giving a directed edge from $q$ to $q'$ labeled by $a$. A path in this graph from a state $q$ to a second state $q'$ is labeled by a word $w$ with $\widehat{\delta}(q, w) = q'$.

A language $L$ is regular if and only if $L$ is the language accepted by some finite state automaton \cite[Section 2.5]{HopUll1979}. We will use both the definition of regular languages and the equivalent idea of languages accepted by FSAs throughout this paper. We will also frequently use synchronously regular languages:

\begin{defn}\cite[Definition 1.4.4]{ECHLPT1992}
Given a language $L \subseteq A_1^\ast \times \cdots \times A_n^\ast$ and padding symbols $\$_i \not \in A_i$ for each $i$, we define a language $L^p$ over the padded alphabet $B = (A_1 \cup \set{\$_1}) \times \cdots \times (A_n \cup \set{\$_n})$ as follows: 
\begin{itemize}
\item For each $n$-tuple $(w_1, \dots, w_n) \in L$, let $m = \text{max}\set{\abs{w_i} \mid 1 \leq i \leq n}$. 
\item We pad each $w_i$ with $\$_i$'s at the end to make its length $m$. 
\item The resulting $n$-tuple of strings is the \emph{padded tuple}, denoted $(w_1, \dots, w_n)^p$.
\item $L^p = \set{(w_1, \dots, w_n)^p \in B \mid (w_1, \dots, w_n) \in L}$
\end{itemize}
$L^p$ is the \emph{padded extension of $L$}. A language $L$ over $(A_1, \dots, A_n)$ is \emph{synchronously regular} if $L^p$ is a regular language over $B$.
\end{defn}

When we consider products of languages, we will consider them as padded languages unless otherwise specified. When we consider the concatenation of two languages over a product alphabet $L \cdot K$, we will concatenate each entry of the tuple before padding and subsequently move any pre-existing padding to the end; that is, $ L^p \cdot K^p = (L \cdot K)^p = \set{(u_1v_1, \dots, u_nv_n)^p \mid (u_1, \dots, u_n) \in L \text{ and } (v_1, \dots, v_n) \in K}$.

\begin{lemma}(Pumping Lemma, \cite[Lemma 3.1 and Exercise 3.2]{HopUll1979}) Let $L$ be a regular language. Then there exists a natural number $n$ such that for any word $z$ with $\abs{z} \geq n$ and any words $p, s$ with $pzs \in L$, we can decompose $pzs =_{A^\ast} puvws$ with $\abs{uv} \leq n$, $\abs{v} \geq 1$, and for all $i \geq 0$ the word $puv^iws$ is an element of $L$. In particular, given an FSA $M$ accepting $L$, we can take $n$ to be the number of states of $M$. \end{lemma}

The smallest such $n$ is the \emph{pumping number} of $L$.

\subsection{Almost Convexity}
\begin{defn} \label{almCon} \cite{Cannon1987} A group $G$ is \emph{almost convex} with respect to a finite generating set $A$ if there is a constant $k$ such that for all $n \in \bb{N}$ and $g, h$ in the sphere $S(n)$ (in $\Gamma_{G, A}$ centered at $1_G$) with $d_{G, A}(g, h) \leq 2$, there exists a path inside the ball $B(n)$ (in $\Gamma_{G, A}$ centered at $1_G$) of length at most $k$ from $g$ to $h$. \end{defn}

Almost convexity was introduced by Cannon \cite{Cannon1987}. Thiel showed that almost convexity is dependent on generating set \cite{Thiel1992}. Elder showed that given a pair $(G, A)$ with FFTP, $G$ is almost convex with respect to $A$ \cite{Elder2005}.

\section{Proofs of Theorems A, B, and C} \label{ThmBC}
In this section, we prove that every weight non-increasing regular CP-RS gives a bounded regular CP-RS, which yields an autostackable structure with Proposition \ref{brcprs}.

\setcounter{thm}{0}
\begin{thm} Suppose $G = \langle A \rangle$ has a weight non-increasing synchronously regular CP-RS $R$. Then $G$ has an autostackable structure which has the same normal forms as $R$. \end{thm}

\begin{proof}
Let $R$ be a length non-increasing synchronously regular CP-RS for $G = \langle A \rangle$ with normal form set $\cali{N}$. Let $$R' = \set{(u, v) \mid (u, v) \in R \text{ and } u \text{ is minimally reducible}}$$ Then $R'$ is terminating, as an infinite sequence of rewritings in $R'$ would necessarily be an infinite sequence of rewritings in $R$, and $R'$ is weight non-increasing. Further, $R' = R \cap (\cali{N}A \times A^\ast)$, an intersection of synchronously regular languages, hence $R'$ is regular. We have $R' \subseteq R$, so for all $(u, v) \in R'$, we have $u =_G v$, hence $G$ is a quotient of $Mon \langle A \mid R' \rangle$. For any word $u$ which can be reduced with $R$, we see that $u =_{A^\ast} u_1u_2$ for some $u_1, u_2 \in A^\ast$ with $u_1 \in \cali{N}$ and $u_1\text{first}(u_2) \not \in \cali{N}$. But $u_1\text{first}(u_2)$ can be rewritten over $R$, hence $u_1\text{first}(u_2)$ is the left side of a pair in $R'$ and is reducible in $R'$. Therefore, the set of irreducible words over $R'$ is the same as the set of irreducible words over $R$, so $R'$ has exactly one irreducible word over $A$ for each element of $G$; this, along with the fact that $G$ is a quotient of $Mon \langle A \mid R' \rangle$, gives that $G = Mon \langle A \mid R' \rangle$, and that the normal forms of $G$ over $R'$ are the same as the normal forms over $R$. Thus $R'$ is a weight non-increasing synchronously regular CP-RS for $G$ with normal form set $\cali{N}$.

Let $M$ be an FSA accepting the padded extension of $R'$, and let $k$ be the number of states in $M$. Let $N = 2k * \max\set{wt(a) \mid a \in A}$.  Let 
$$S_1 = \set{(u, \text{nf}_{R'}(u)) \mid u \text{ is minimally reducible and } wt(u) < N}$$
 and let 
\begin{align*}
S_2 = \set{(u_1u_2l, u_1\text{nf}_{R'}(u_2lv_2^{-1})v_2) \mid &(u_1u_2l, v_1v_2) \in R', wt(u_2) < N,\\ &wt(u_2l) \geq N, \text{ and } v_1 = v(\abs{u_1})}\\ 
\end{align*}
(A visual representation of the rules in $S_2$ in the Cayley graph $\Gamma_{G, A}$ is given in Figure \ref{fig: S2}.) Then let $S = S_1 \cup S_2$. We claim that $S$ is a bounded regular convergent prefix-rewriting system for $G$ with normal forms $\cali{N}$. For this, we must prove five things: that $G = Mon \langle A \mid \set{u = v \mid (u, v) \in S} \rangle$; that $S$ is bounded; that $S$ is synchronously regular; that $S$ is terminating; and that $\cali{N}$ is the set of irreducible words over $S$. 

\begin{figure}
\begin{tikzpicture}[scale = 2, every node/.style={circle, draw, fill= black, scale = 0.3}]
\node(1) at (0, 0)[label = left:\fontsize{40}{40}{$1_G$}]{};
\node(u1u2) at (5, 0.5){};
\node(v1v2) at (5, 0.25){};
\node(u1) at (4, 0.75){};
\node(v1) at (4, -0.25){};
\draw [rounded corners, -{Latex[length = 3mm, width = 2mm]}] (1) -- (2, -0.25) -- (v1) node[coordinate, behind path, midway, label = below:\fontsize{40}{40}{$v_1$}] {};
\draw [rounded corners, -{Latex[length = 3mm, width = 2mm]}] (1) -- (1, 0.5) -- (2, 0.5) -- (u1) node[coordinate, behind path, midway, label = above:\fontsize{40}{40}{$u_1$}] {};
\draw [rounded corners, -{Latex[length = 3mm, width = 2mm]}] (u1) -- (4.5, 0.6) -- (u1u2) node[coordinate, behind path, midway, label = above:\fontsize{40}{40}{$u_2$}] {};
\draw [decorate, decoration = random steps, -{Latex[length = 3mm, width = 2mm]}] (u1) -- (v1) node[coordinate, behind path, midway, label = left:\fontsize{40}{40}{$\text{nf}_{R'}(u_2lv_2^{-1})$}] {};
\draw [rounded corners, -{Latex[length = 3mm, width = 2mm]}] (v1) -- (4.5, -0.25) -- (v1v2) node[coordinate, behind path, midway, label = below:\fontsize{40}{40}{$v_2$}] {};
\draw [-{Latex[length = 3mm, width = 2mm]}] (u1u2) -- (v1v2) node[coordinate, behind path, midway, label = right:\fontsize{40}{40}{$l$}] {};
\end{tikzpicture}
\caption{For pairs in $S_2$, the left word follows the top path from $1_G$ to $u_1u_2l$, while the right word reroutes via the $\text{nf}_{R'}(u_2lv_2^{-1})$ bridge.}
\label{fig: S2}
\end{figure}
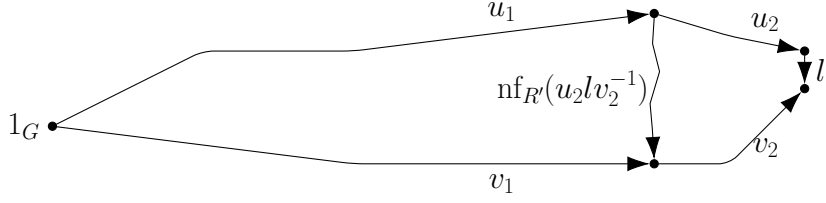

\subsection*{Prefix-rewriting system for $G$:}
We first prove that $$G = Mon \langle A \mid \set{u = v \mid (u, v) \in S} \rangle.$$ We prove this by showing that each relation given by $S$ is a relation in $G$, and that each relation in $R'$ (which we know is a prefix-rewriting system for $G$) is a result of relations in $S$.

We first consider relations $uw =_S vw$ arising from pairs $(u, v)$ in $S_1$: these are relations in $G$, since $uw =_G \text{nf}_{R'}(u)w =_{A^\ast} vw$. Relations arising from pairs $(u, v)$ in $S_2$ are also relations in $G$, as 
\begin{align*}
uw &=_{A^\ast} u_1u_2lw \\
&=_G u_1u_2lv_2^{-1}v_2w \\
&=_{R'} u_1\text{nf}_{R'}(u_2lv_2^{-1})v_2w \\
&=_{A^\ast} vw
\end{align*}
Since $R'$ is a prefix-rewriting system for $G$, we see that the third equality is also true in $G$, so $u =_G v$. Thus, equality of two words over $S$ implies equality of the same two words in $G$.

Now, let $(u, v)$ be a pair from $R'$. Our next goal is to show that $u =_S v$. We proceed by induction on $wt(u)$. Note that the set $\set{wt(u) \mid u \in A^\ast}$ is the set of non-negative integer combinations of finitely many positive values, hence is well-founded, so induction is a viable proof technique here.

For our base case, suppose that $wt(u) < N$. Then $wt(v) \leq wt(u)$ because $R'$ is weight non-increasing, so $S_1$ contains the pair $(u, \text{nf}_{R'}(u))$, and either $v =_{A^\ast} \text{nf}_{R'}(v)$ or $S_1$ contains the pair $(v, \text{nf}_{R'}(v))$. Since $u =_G v$, we have that $\text{nf}_{R'}(u) =_{A^\ast} \text{nf}_{R'}(v)$, and we have both $u =_S \text{nf}_{R'}(u)$ and $v =_S \text{nf}_{R'}(u)$, so $u =_S v$.

Now suppose that $wt(u) = n \geq N$, and that whenever $u' =_{R'} v'$ with $wt(u') < n$ and $wt(v') < n$, we have that $u' =_S v'$. Then let $l = \text{last}(u)$, let $u_2l$ be the shortest suffix of $u$ with weight at least $N$, and let $u_1$ be the prefix of $u$ such that $u = u_1u_2l$. If $\abs{v} \geq \abs{u_1}$, let $v_1 = v(\abs{u_1})$, and $v_2$ be the suffix of $v$ such that $v =_{A^\ast} v_1v_2$; otherwise, let $v_1 = v(\abs{v_1} - 1)$ and let $v_2$ be the suffix of $v$ with length $1$. We need to show that $u_1u_2l =_S v_1v_2$. We see that the pair $(u_1u_2l, u_1\text{nf}_{R'}(u_2lv_2^{-1})v_2)$ is an element of $S$, and that $\text{nf}_{R'}(u_1\text{nf}_{R'}(u_2lv_2^{-1})) =_{A^\ast} \text{nf}_{R'}(v_1)$. 

Consider $wt(\text{nf}_{R'}(u_2lv_2^{-1}))$. Recall that $M$ is the finite state automaton accepting all padded pairs from $R'$. Also recall from Definition \ref{FSA} that $\widehat{\delta}(q, \lambda) = q$ for all states $q$, and $\widehat{\delta}(q, aw) = \widehat{\delta}(q, aw) = \widehat{\delta}(\delta(q, a), w)$; that is, $\widehat{\delta}(q, w)$ is the state of $M$ after starting at the state $q$ and reading the word $w$. Suppose that $\abs{v_1} = \abs{u_1}$; then we see that $q = \widehat{\delta}(q_0, (u_1, v_1))$ is a state in $M$ from which we can reach a state in $P$ --- in particular, $\widehat{\delta}(q, (u_2l, v_2)^p) \in P$ since $(u_1u_2l, v_1v_2)^p$ is accepted by $M$. But since $M$ has only $k$ states, and we can replace any path in $M$ by a path which does not repeat states, we can replace this path by one of length at most $k - 1$ (potentially touring through every state, and using $k - 1$ edges in total). This shorter path corresponds to a pair of words $(u_2'', v_2'')$ such that $(u_1u_2'', v_1v_2'') \in R'$ and $\abs{u_2''} \leq k - 1$ and $\abs{v_2''} \leq k - 1$. Then we see that $u_1u_2'' =_G v_1v_2''$, hence $u_2''v_2''^{-1} =_G u_1^{-1}v_1$. We also have $u_1u_2l =_G v_1v_2$, hence $u_2lv_2^{-1} =_G u_1^{-1}v_1$. Combining these gives us $u_2''v_2''^{-1} =_G u_2lv_2^{-1}$, so $\abs{\text{nf}_{R'}(u_2lv_2^{-1})} \leq \abs{u_2''v_2''^{-1}} \leq 2k - 2$. On the other hand, suppose that $\abs{v_1} < \abs{u_1}$. Then $\abs{v_2} = 1$, and $\widehat{\delta}(q_0, (u_1, v)^p)$ is a state in $M$ from which we can reach a state in $P$. This again gives a path in $M$ of length at most $k - 1$, corresponding to a path in $\Gamma_{G, A}$ from $\overline{u_1}$ to $\overline{v}$ of length at most $2k - 2$, which can be extended with a single edge to a path from $\overline{u_1}$ to $\overline{v_1}$ of length at most $2k - 1$. Either way, we have \begin{equation}\label{eq:2kineq}\abs{nf_{R'}(u_2lv_2^{-1})} \leq 2k - 1 \end{equation} In particular, this gives us that $wt(nf_{R'}(u_2lv_2^{-1}) \leq (2k - 1) * \max\set{wt(a) \mid a \in A} < N$. We can now use our induction hypothesis: $wt(u_1nf_{R'}(u_2lv_2^{-1})) < wt(u_1u_2l) = n$, so $u_1nf_{R'}(u_2lv_2^{-1}) =_S nf_{R'}(u_1nf_{R'}(u_2lv_2^{-1}))$. Similarly, $wt(v_1v_2) \leq n$, and since $\abs{v_2} \geq 1$ we have $wt(v_1) < n$, so by our induction hypothesis $v_1 =_S nf_{R'}(v_1)$. Since $v_1 =_G u_1nf_{R'}(u_2lv_2^{-1})$, we have $nf_{R'}(v_1) =_{A^\ast} nf_{R'}(u_1nf_{R'}(u_2lv_2^{-1}))$. Stringing this all together, we have
\begin{align*}
u &=_{A^\ast} u_1u_2l \\
&=_S u_1\text{nf}_{R'}(u_2lv_2^{-1})v_2 \\
&=_S \text{nf}_{R'}(u_1\text{nf}_{R'}(u_2lv_2^{-1}))v_2 \\
&=_{A^\ast} \text{nf}_{R'}(v_1)v_2 \\
&=_S v_1v_2 \\
&=_{A^\ast} v
\end{align*}
Thus, equality of two words over $R'$ implies equality of the same two words over $S$, so equality of two words in $G$ implies equality over $S$. Combining this with the fact that equality over $S$ implies equality in $G$ (proved above), we see that $S$ is a prefix-rewriting system for $G$.

\subsection*{Bounded:}
We now consider boundedness. Each rule in $S_1$ rewrites a word of weight at most $N$ (and thus length at most $\frac{n}{\min\set{wt(a) \mid a \in A}}$) to another word of weight at most $N$, hence each rule in this subset of $S$ rewrites a substring of bounded length to a substring of bounded length.

Now, consider a pair $(u_1u_2l, u_1\text{nf}_{R'}(u_2lv_2^{-1}v_2)$ from $S_2$. We note that since $wt(u_2) < N$, we have $\abs{u_2} < \frac{N}{\min\set{wt(a) \mid a \in A}}$, so $\abs{u_2l}$ is bounded. By equation \ref{eq:2kineq} above, $\abs{\text{nf}_{R'}(u_2lv_2^{-1})} \leq 2k - 1$. Suppose for sake of contraditction that $\abs{v_2} > \abs{u_2l} + 1 + (k - 1)$; then let $v_3 = v_2(\abs{u_2l})$, and let $v_4$ be the suffix of $v_2$ such that $v_2 = v_3v_4$. Then the path in $M$ starting at $\widehat{\delta}(q_0, (u, v_1v_3))$ and labeled by $(\lambda, v_4)^p$ has length at least $k$, and must repeat a state. This gives us a decomposition $v_4 = pzs$ with $\abs{z} > 0$ such that $(u, v_1v_3pz^ns)^p$ is accepted by $M$ for all $n$, which is impossible since $v_1v_3pz^ns$ has weight larger than $wt(u)$ for sufficiently large $n$. Thus, $\abs{v_2} \leq \abs{u_2l} + k$. Thus, each rule in this subset rewrites a substring of bounded length to a substring of bounded length, so $S$ is bounded.

\subsection*{Synchronously regular:}
We next consider regularity. Let $D = \set{(a, a) \mid a \in A}$. Let $$L_{(u_2l, v_2)} = \left((R')^{p} \cap \left((A \times A)^\ast \cdot (\set{u_2l} \times \set{v_2})^{p}\right)\right) / (\set{u_2l} \times \set{v_2})^{p}$$
for each $l \in A$ and $u_2l, v_2 \in A^{\leq k + 2}$. We note that $L_{(u_2l, v_2)}$ is $\set{(u_1, v_2) \mid (u_1u_2l, v_1v_2) \in R'}$. We then let $P_{(u_2l, v_2)} = \pi_1(L_{(u_2l, v_2)})$ where $\pi_1$ is the projection map onto the first coordinate. Then $P_{(u_2l, v_2)}$ is the set of all $u_1$ such that there exists some $v_1$ with $(u_1u_2l, v_1v_2) \in R'$. Then $S_2$ is 
$$\bigcup_{u_2l \in A^{\leq k + 2}} \bigcup_{v_2 \in A^{\leq k + 2}} \left((P_{(u_2l, v_2)} \times P_{(u_2l, v_2)}) \cap D^\ast \right)^{p} \cdot (\set{u_2l} \times \set{\text{nf}_{R'}(u_2lv_2^{-1})v_2})^{p}$$
Because synchronously regular languages are closed under finite unions, finite intersections, products, projections, quotients, Kleene stars, and concatenation on the right by finite languages, and $S_2$ is built from regular languages (namely $R'$, $D$, and several finite languages) using finitely many of these operations, $S_2$ is synchronously regular. Since $S_1$ is finite, $S_1$ is also synchronously regular. Thus, $S$ is a union of two synchronously regular languages, hence $S$ is synchronously regular.

\subsection*{Terminating:}
We now consider termination. We define a partial order $\prec$ on directed edges in $\Gamma_{G, A}$ as follows: For an edge $e$, define $w(e) = \text{nf}_{R'}(\iota(e))\text{label}(e)$. Then $e' \prec e$ if $wt(w(e')) < wt(w(e))$, or $wt(w(e')) = wt(w(e))$ and there exists some rewriting sequence in $R'$ $$w(e) = w_0 \to w_1 \to \cdots \to w_n = w(e')$$ Because there is no infinite chain of rewritings over $R'$, any infinite descending chain $e_0 \succ e_1 \succ e_2 \succ \cdots$ must have some $i$ with $wt(e_i) < wt(e_0)$; repeating this argument gives a subchain $e_0 = e_0' \succ e_1' \succ e_2' \succ \cdots$ with $wt(e_i') < wt(e_{i - 1}')$, contradicting the well-foundedness of the natural numbers. Thus, $\prec$ admits no infinite descending chains, hence $\prec$ is well-founded.

Now, we extend this order on edges to an order on words over $A^\ast$ as follows: Assign to each word $u$ the set of directed edges $E_u$ contained in the path in $\Gamma_{G, A}$ starting at $1_G$ and labeled by $u$. Then $u' < u$ if there exists some $e \in E_u$ such that for all $e' \in E_{u'}$, $e' \prec e$. The relation $<$ is transitive, since whenever we have $u < v$ and $v < w$, there exist $e_v \in E_v$ and $e_w \in E_w$ with the property that for all $e' \in E_u$, $e' \prec e_v \prec e_w$, hence $e' \prec e_w$. The relation $<$ is also antisymmetric, since $u < v$ and $v < u$ would give the existence of edges $e_u \in E_u$ and $e_v \in E_v$ with $e_u \prec e_v$ and $e_v \prec e_u$, contradicting antisymmetry of $\prec$. Further, $<$ is irreflexive, since for any $e_u \in E_u$, we have $e_u \not \prec e_u$, hence $u \not < u$. Thus, $<$ is a strict partial order on $A^\ast$. Further, any infinite descending chain $u_0 > u_1 > u_2 > \cdots$ would give an infinite descending chain $e_0 \succ e_1 \succ e_2 \succ \cdots$ with $e_i \in E_{u_i}$, violating well-foundedness of $\prec$; thus, $<$ is well-founded.

We now show that $S$ decreases $<$; that is, for every $(u, v) \in S$, we have $v < u$. We have two cases to consider:

Suppose $(u, \text{nf}_{R'}(u)) \in S_1$. Let $l = \text{last}(u)$ and $u_1$ be the prefix of $u$ such that $u_1l = u$. Then let $e$ be the edge in $\Gamma_{G, A}$ starting at $u_1$ and labeled by $l$. Then we have $w(e) =_G u$, hence $wt(w(e)) \geq wt(\text{nf}_{R'}(u))$. This gives that for all $e' \in E_{\text{nf}_{R'}(u)}$ we have $wt(w(e')) < wt(w(e))$, with possibly one exception: the final edge of the path in $\Gamma_{G, A}$ starting at $1_G$ that is labeled by $\text{nf}_{R'}(u)$. As such, $e' \prec e$ for all but possibly this final edge. Now, let $e'$ be the final edge of the path in $\Gamma_{G, A}$ starting at $1_G$ that is labeled by $\text{nf}_{R'}(u)$. We notice that because normal forms of $R'$ are closed under prefixes, we have $w(e') = \text{nf}_{R'}(u)$. Then there is some rewriting sequence in $R'$ taking $w(e)$ to $\text{nf}_{R'}(w(e)) =_{A^\ast} \text{nf}_{R'}(u) =_{A^\ast} w(e')$, so $e' \prec e$. Thus, all edges in $E_{\text{nf}_{R'}(u)}$ are smaller than $e \in E_{u}$, so $\text{nf}_{R'}(u) < u$, as desired.

Now, suppose $(u_1u_2l, u_1\text{nf}_{R'}(u_2lv_2^{-1})v_2) \in S_2$. Let $e$ be the edge labeled by $l$ starting at $u_1u_2$, and let $e' \in E_{u_1\text{nf}_{R'}(u_2lv_2^{-1})v_2}$. Note that $w(e) = u_1u_2l$, since $\text{nf}_{R'}(\iota(e)) =_{A^\ast} \text{nf}_{R'}(\overline{u_1u_2}) =_{A^\ast} u_1u_2$ and $\text{label}(e) = l$. We have four cases, depending on the location of $e'$ in the path starting at $1_G$ and labeled by $u_1\text{nf}_{R'}(u_2lv_2^{-1})v_2$:
\begin{enumerate}[(i)]
\item Suppose $e'$ is an edge in the subpath starting at $1_G$ and labeled by $u_1$. Then $w(e')$ is a prefix of $u_1$ since $u_1$ is a normal form of $R'$ and normal forms of $R'$ are closed under prefixes. Thus $w(e')$ is a proper prefix of $w(e)$, so $wt(w(e')) < wt(w(e))$, showing that $e' \prec e$.
\item Suppose $e'$ is an edge in the subpath starting at $u_1$ and labeled by $\text{nf}_{R'}(u_2lv_2^{-1})$. Then $wt(w(e')) \leq wt(u_1) + 2k * \max\set{wt(a) \mid a \in A} = wt(u_1) + N < wt(u_1) + wt(u_2l) = \abs{w(e)}$, so $e' \prec e$.
\item Suppose $e'$ is an edge in the path starting at $u_1\text{nf}_{R'}(u_2lv_2^{-1})$ and labeled by $v_2$, but is not the last edge of this path. Then $wt(w(e')) < wt(v_1v_2) \leq wt(u) = wt(w(e))$, hence $e' \prec e$.
\item Finally, suppose $e'$ is the final edge of the path starting at $u_1\text{nf}_{R'}(u_2lv_2^{-1})$ and labeled by $v_2$. Then $wt(w(e')) \leq wt(v_1v_2) \leq wt(u) = wt(w(e))$. Moreover, there is a single rewriting from $R'$ which rewrites $w(e) =_{A^\ast} u_1u_2l$ to $v_1v_2$. Let $v =_{A^\ast} v_1v_2$, $l' = \text{last}(v)$ and $v'$ be the prefix of $v$ such that $v'l' =_{A^\ast} v$. Then there is a sequence of rewritings from $R'$ taking $v'$ to $\text{nf}_{R'}(v')$. Appending this rewriting sequence to our rewriting from $w(e)$ to $v_1v_2$ gives a rewriting sequence from $R'$ that starts at $w(e)$ and ends at $\text{nf}_{R'}(v')\text{last}(v) =_{A^\ast} w(e')$. Thus, $e' \prec e$.
\end{enumerate}

In all cases, we have $e' \prec e$ (and, in particular, $d_{G, A}(1_G, \iota(e')) \leq d_{G, A}(1_G, \iota(e))$, which will be useful in the proof of Theorem \ref{extraToGeo}); thus, $u_1\text{nf}_{R'}(u_2lv_2^{-1})v_2 < u_1u_2l$, as desired. So $S$ decreases a well-founded partial ordering, hence $S$ is terminating.

\subsection*{Normal forms}
Finally, we show that $S$ has a set of unique normal forms, namely the normal forms over $R'$.

Suppose that $w$ is reducible over $R'$. Then $w$ has some maximal prefix $u$ satisfying that $u$ is irreducible over $R'$. Since $u$ is a proper prefix, there exists some $l \in A$ such that $ul$ is a prefix of $w$. Then $ul$ is not irreducible over $R'$, hence $w$ has a prefix which is minimally reducible over $R'$. Then there is some rule $(ul, v)$ in $R'$ which we can apply to $w$. Further, if $wt(ul) < N$, then $(ul, \text{nf}_{R'}(ul))$ is a rule in $S_1$, so $w$ is also reducible over $S$; and if $wt(ul) \geq N$, then $ul$ is the left-hand side of a rule in $S_2$, so $w$ is reducible over $S$. Either way, we have that $w$ is reducible over $S$.

Alternatively, suppose that $w$ is reducible over $S$. Then some prefix $v$ of $w$ is the left-hand side of a rule from $S_1$ or from $S_2$, so $v$ is reducible over $R'$. As $v$ is a prefix of $w$, this gives us that $w$ is also reducible over $R'$.

Thus, the words which are reducible over $R'$ are the same as the words which are reducible over $S$. This gives that the irreducible words over both rewriting systems are the same, so $S$ has a set of unique normal forms, and $\cali{N}_S = \cali{N}_{R'} = \cali{N}_R$.

With all of the above, we see that $S$ is a bounded regular convergent prefix-rewriting system for $G$. Applying Proposition \ref{brcprs}, we see that $G$ is autostackable with normal form set $\cali{N}$.

\end{proof}

Note that, since any length non-increasing CP-RS is a weight non-increasing CP-RS with each generator having length 1, this theorem shows that any group with a length non-increasing CP-RS is autostackable. We now prove the following theorem as an extension to Theorem \ref{lniToAutost} when we have additional restrictions on $R$:

\setcounter{thm}{2}
\begin{thm} Suppose that $G = \langle A \rangle$ admits a length non-increasing \propL synchronously regular CP-RS $R$. Then the autostackable structure for $G$ constructed in the proof of Theorem \ref{lniToAutost} is a geodesically autostackable structure. \end{thm}

\begin{proof}
Suppose $G = \langle A \rangle$ admits a length non-increasing \propL synchronously regular CP-RS $R$. Let $S$ be the bounded regular CP-RS for $G$ constructed in the proof of Theorem \ref{lniToAutost}, and let $\Phi$ be the flow function constructed from $S$ as in the proof of \cite[Theorem 5.3]{BHH2014}. That is, for each group element $y \in G$ and each letter $a \in A$, $$\Phi(e_{y, a}) = \begin{cases} e_{y, a} & \text{if } \text{nf}_R(y)a \in \cali{N}_R \text{ or } \text{last}(\text{nf}_R(y)) = a^{-1} \\ p_{y, s^{-1}t} & \text{otherwise}\end{cases}$$ where $p_{y, s^{-1}t}$ is the path starting at the vertex corresponding to $y$ and labeled by the word $s^{-1}t$, with $(\text{nf}(y)a, y') \in S$, $\text{nf}(y) =_{A^\ast} ws$, $y' =_{A^\ast} wt$, and the words $s$ and $t$ do not start with the same letter. 
Then each directed edge $e$ in $\Gamma_{G, A}$ falls into one of the following cases:
\begin{itemize}
\item $\text{nf}_S(\iota(e))\text{label}(e) \in \cali{N}_R$. In this case, $\Phi(e) = e$.
\item $\abs{\text{nf}_S(\iota(e))\text{label}(e)} \leq k + 1$ and $\text{nf}_S(\iota(e))\text{label}(e) \not \in \cali{N}_R$. In this case, $\text{label}(\Phi(e)) = s^{-1}s'$ for some suffix $s$ of $\text{nf}_S(\iota(e))$ and some suffix $s'$ of $\text{nf}_S(\iota(e)\text{label}(e))$. Recall from Definition \ref{geoAutost} that $\alpha(e) = \frac{1}{2}\left( d_{G, A}(1_G, \iota(e)) + d_{G, A}(1_G, \tau(e))\right)$. For each edge $e'$ in the path $\Phi(e)$ except for possibly the final edge, we have that $e'$ is an edge in the path starting at $1_G$ and labeled by a geodesic with length at most $\abs{\text{nf}_S(\iota(e))}$, hence $\alpha(e') < \abs{\text{nf}_S(\iota(e))} \leq \alpha(e)$. For the final edge $e_f$ of the path $\Phi(e)$, we have that $e_f$ is a portion of the normal form of $\iota(e)\text{label}(e)$, so $\Phi(e_f) = e_f$.
\item $\abs{\text{nf}_S(\iota(e))\text{label}(e)} > k + 1$ and $\text{nf}_S(\iota(e))\text{label}(e) \not \in \cali{N}_R$. In this case, $\text{label}(\Phi(e))$ is the word $w$ obtained by freely reducing $u_2^{-1}\text{nf}_S(u_2lv_2^{-1})v_2$ for some $u_2, v_2$, and $l$ as in the definition of $S_2$. Each edge in the path starting at $\iota(e)$ and labeled by $w$ is also an edge in the path starting at $\iota(e)$ and labeled by $u_2^{-1}\text{nf}_S(u_2lv_2^{-1})v_2$. As we proved in the termination subsection of the proof of Theorem \ref{lniToAutost}, each edge $e'$ in the path starting at $\iota(e)$ and labeled by $u_2^{-1}\text{nf}_S(u_2lv_2^{-1})v_2$ has $d_{G, A}(1_G, \iota(e')) \leq d_{G, A}(1_G, \iota(e))$ and $d_{G, A}(1_G, \tau(e')) \leq d_{G, A}(1_G, \tau(e))$, with equality only at possibly the final edge $e_f$ of this path. From this, we see that $\alpha(e') < \alpha(e)$ except when $e' = e_f$. But $e_f$ is the edge labeled by $\text{last}(\text{nf}_S(\iota(e)\text{label}(e))$ with endpoint $\iota(e)\text{label}(e)$, so $e_f$ is a portion of the normal form of $\iota(e)\text{label}(e)$, giving $\Phi(e_f) = e_f$.
\end{itemize}

In all three cases, we have that whenever $e'$ is an edge in the path $\Phi(e)$, we have either $\alpha(e') < \alpha(e)$ or $\Phi(e') = e'$. Thus, $\Phi$ is the flow function for a geodesically autostackable structure.
\end{proof}

We can also use the proof of Theorem \ref{lniToAutost} to prove almost convexity in the following theorem. The proof is similar to those found in \cite[Theorem B]{HM1997} and \cite[Theorem 4.4]{BH2015} that groups with geodesic finite complete rewriting systems and geodesically stackable groups (respectively) are almost convex. We include the details of the proof for sake of completeness.

\setcounter{thm}{1}
\begin{thm} Suppose $G = \langle A \rangle$ has a length non-increasing synchronously regular CP-RS. Then $G$ is almost convex with respect to $A$. \end{thm}

\begin{proof}
Suppose $G = \langle A \rangle$ has a length non-increasing synchronously regular CP-RS. Define a bounded regular CP-RS $S$ for $G = \langle A \rangle$ as in the proof of Theorem \ref{lniToAutost}, and let $g, h \in G$ with $g, h \in S(n)$ and $d_{G, A}(g, h) \leq 2$.

In the case that $d_{G, A}(g, h) = 1$, let $a_0 \in A$ such that $ga_0 =_G h$. Then $\text{nf}_S(g) = w_0u_0$ with $(w_0u_0a_0, w_0v_0a_1)$ being a rule in $S$ for some $w_0, u_0, v_0 \in A^\ast$, and $a_1 \in A$. As a consequence of the proof of termination in the proof of Theorem \ref{lniToAutost}, the path starting at $\overline{w_0}$ and labeled by $v_0$ lies entirely within $B(n)$. If $a_1 = \text{last}(\text{nf}_S(h))$, then $e_{ha_1^{-1}, a_1}$ lies in $B(n)$, and we see that the path labeled by $u_0^{-1}v_0a_1$ starting at $g$ ends at $h$, has length at most $2k + 2 + 2k + 2 + 1 = 4k + 5$, and lies entirely within $B(n)$. Otherwise, we can repeat this process: $w_0v_0a_1 =_G \text{nf}_S(w_0v_0)a_1 =_{A^\ast} w_1u_1a_1$, with some rule $(w_1u_1a_1, w_1v_2a_2)$ in $S$. In this way, we get a chain of equalities and rewritings $$w_0u_0a_0 \to_S w_0v_0a_1 =_G w_1u_1a_1 \to_S \cdots \to_S w_jv_ja_{j + 1}$$ where $a_{j + 1} = \text{last}(\text{nf}_S(h))$, and $\abs{w_iu_i} \leq n$ and $\abs{w_iv_i} \leq n$ for all $i$. This gives a path from $g$ to $h$ labeled by $u_0^{-1}v_0u_1^{-1}v_1\dots u_j^{-1}v_ja_{j + 1}$ which lies entirely within $B(n)$. Each $u_i^{-1}v_i$ piece of the path has length at most $4k + 4$ (from the proof of boundedness in Theorem \ref{lniToAutost}). We cannot repeat any $a_i$ because having $a_i = a_k$ for some $k > i$ would give a loop of rewritings $w_iu_ia_i \to_S \cdots \to_S w_ku_ka_k$, where both $w_iu_i$ and $w_ku_k$ are the unique normal form of $ha_i^{-1}$. Thus, there are at most $\abs{A}$ pieces in our path of the form $u_i^{-1}v_i$, plus the final edge $a_{j + 1}$. In this case, we have a path in $B(n)$ from $g$ to $h$ of length at most $(4k + 4)\abs{A} + 1$.

Now, we consider the case that $d(g, h) = 2$. Then $h = gab$ for some $a, b \in A$. There are three subcases. If $d(1_G, ga) = n - 1$, we have a path of length 2 from $g$ to $h$ lying within $B(n)$, namely the path starting at $g$ and labeled by $ab$. If $d(1_G, ga) = n$, we can apply the distance 1 case twice, giving a path of length at most $2(4k + 4)\abs{A}$ from $g$ to $h$ lying within $B(n)$. This leaves the case where $d(1_G, ga) = n + 1$. Let $c = \text{last}(\text{nf}_S(ga))$, and $g' = gac^{-1}$. We now need to provide a path in $B(n)$ from $g$ to $g'$ of bounded length, and can repeat the process to make a path from $g'$ to $h$. If $g = g'$, then the path of length 0 from $g$ to $g'$ lies within $B(n)$, so suppose instead that $g \neq g'$. Then $\text{nf}_S(g)a$ is not itself a normal form, hence $\text{nf}_S(g)a = w_0u_0a_0$; this rewrites to $w_0v_0a_1$, and again $v_0$ lies entirely within $B(n)$, by the same argument as the distance 1 case. Again, we have a chain of equalities and rewritings $$w_0u_0a_0 \to w_0v_0a_1 =_G w_1u_1a_1 \to \cdots w_jv_ja_{j + 1}$$ where $a_{j + 1} = \text{last}(\text{nf}_S(g'))$, and each $w_i, u_i, v_i \in B(n)$. This again gives a path of length at most $(4k + 4)\abs{A}$ from $g$ to $g'$ lying within $B(n)$. Repeating the process for a path from $g'$ to $h$ gives a path from $g$ to $h$ of length at most $2(4k + 4)\abs{A}$.

Thus, we have a path of length at most $2(4k + 4)\abs{A}$ from $g$ to $h$ in $B(n)$ whenever $g, h \in S(n)$ and $d_{G, A}(g, h) \leq 2$, so $G$ is almost convex with respect to $A$.
\end{proof}

\section{Proof of Theorem D} \label{ThmA}
In this section, we produce a length non-increasing regular CP-RS with short reverse lexicographic normal forms for a pair $(G, A)$ with FFTP. We begin with a lemma that will allow us to handle geodesics which are not short reverse lexicographic normal forms:

\begin{lemma}\label{doubleFellowTravel}
Suppose the pair $(G, A)$ has FFTP with fellow traveler constant $k$, $A$ is totally ordered, and $u \in A^\ast$ is a word representing $g \in G$ which is not a short reverse lexicographic normal form. Then there exists some word $v$ with $v =_G u$, $v <_{srev} u$, and $u$ and $v$ $2k$-fellow travel.
\end{lemma}

\begin{proof}
Suppose $(G, A)$ has FFTP with fellow traveler constant $k$, and $u \in A^\ast$ is a word representing $g \in G$ which is not a short reverse lexicographic normal form. If $u$ is not geodesic, we take $v$ to be any witness of $u$; then $v =_G u$, $\abs{v} < \abs{u}$ so $v <_{srev} u$, and $u$ and $v$ $k$-fellow travel, hence also $2k$-fellow travel.

Now, suppose $u$ is geodesic. Then let $w$ be the short reverse lexicographic normal form of $g$, and $u_2$ the longest common suffix of $u$ and $w$, so $u = u_1u_2$ and $w = w_1lu_2$ for some $w_1 \in A^\ast$ and $l \in A$. Then $u_1l^{-1}$ is not a geodesic, having length $\abs{u_1} + 1$ while a geodesic representative of the same element has length $\abs{u_1} - 1$. Because $(G, A)$ has FFTP, there exists some word $v'$ with $\abs{v'} < \abs{u_1} + 1$, $v' =_G u_1l^{-1}$, and $v'$ and $u_1l^{-1}$ $k$-fellow travel. If $\abs{v'} = \abs{u_1} - 1$, we take $v'' = v'$; otherwise, $\abs{v'} = \abs{u_1}$, so there exists some $v''$ with $\abs{v''} = \abs{u_1} - 1$, $v'' =_G v'$, and $v''$ and $v'$ $k$-fellow travel. Now, let $v = v''lu_2$. Then $v$ is short reverse lexicographically smaller than $u$, since the two have the same length, $v$ has a longer common suffix with $w$ than $u$ does with $w$, and $w$ is short reverse lexicographically smaller than $u$. Further, $u$ and $v$ $2k$-fellow travel, since $u_1$ and $v''l$ $2k$-fellow travel and $u$ and $v$ extend these by the same suffix.
\end{proof}

\begin{figure}
\begin{tikzpicture}[scale = 2, every node/.style={circle, draw, fill= black, scale = 0.3}]
\node(1) at (0, 0)[label = left:\fontsize{40}{40}{$1_G$}]{};
\node(u1) at (4, 0){};
\node(u1u2) at (5, 0)[label = right: \fontsize{40}{40}{$g$}]{};
\node(w1) at (3.75, -0.25){};
\draw [decorate, decoration = random steps, -{Latex[length = 3mm, width = 2mm]}] (u1) -- (u1u2) node[coordinate, behind path, midway, label = above: \fontsize{40}{40}{$u_2$}]{};
\draw [decorate, decoration = random steps, rounded corners, -{Latex[length = 3mm, width = 2mm]}] (1) -- (1, 1) -- (2, 1.5) -- (3, 1) -- (u1) node[coordinate, behind path, midway, label = above:\fontsize{40}{40}{$u_1$}] {};
\draw [decorate, decoration = random steps, rounded corners, -{Latex[length = 3mm, width = 2mm]}] (1) -- (1, 0.6) -- (2, 0.7) -- (3, 0.6) -- (w1) node[coordinate, behind path, midway, label = above:\fontsize{40}{40}{$v'$}] {};
\draw [decorate, decoration = random steps, -{Latex[length = 3mm, width = 2mm]}] (1) -- (1, 0.25) -- (2, 0.4) -- (w1) node[coordinate, behind path, midway, label = above:\fontsize{40}{40}{$v''$}] {};
\draw [decorate, decoration = random steps, rounded corners, -{Latex[length = 3mm, width = 2mm]}] (1) -- (1, -0.5) -- (2, -0.75) -- (3, -0.5) -- (w1) node[coordinate, behind path, midway, label = below:\fontsize{40}{40}{$w_1$}] {};
\draw [-{Latex[length = 3mm, width = 2mm]}](w1) -- (u1) node[coordinate, behind path, midway, label = below: \fontsize{40}{40}{$l$}] {};
\end{tikzpicture}
\caption{While $u_1$ and $w_1l$ might be quite far apart, we see that $u_1$ and $v'$ $k$-fellow travel, and similarly $v'$ and $v''$ $k$-fellow travel.}
\label{fig: l41}
\end{figure}
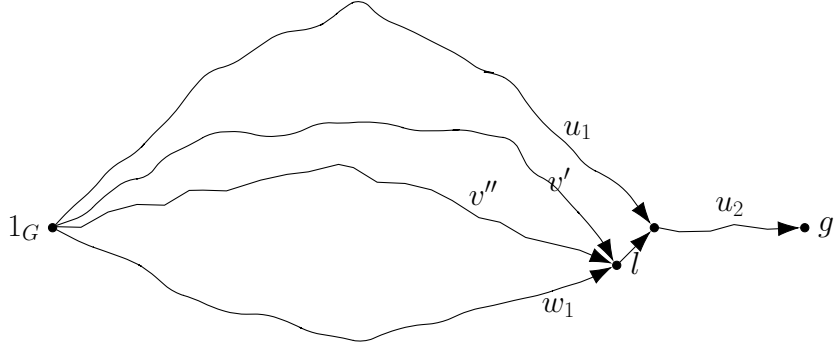

The proof of Lemma $\ref{doubleFellowTravel}$ is illustrated in Figure $\ref{fig: l41}$. We now prove the main theorem.

\setcounter{thm}{3}
\begin{thm}
[Main Theorem] Suppose the pair $(G, A)$ has the falsification by fellow traveler property. Then
\begin{enumerate}[(a)]
\item $G$ admits a length non-increasing \propL synchronously regular CP-RS $R$; and
\item $G$ is geodesically autostackable.
\end{enumerate}
\end{thm}

\begin{proof}
Suppose $(G, A)$ has FFTP with fellow traveler constant $k$. Let $\prec$ be a total ordering on $A$, and let $\cali{N}$ be the set of short reverse lexicographic normal forms of $G$. Let $$L = \set{(u, v) \mid u =_G v, v <_{srev} u, \text{ and } u \text{ and } v \text{ } 4k\text{-fellow travel}}.$$ We create an FSA $M$ accepting $L^p$ as follows:

The alphabet of $M$ is $A \times (A \cup \set{\$})$. For each element $g \in G$ with $\abs{g} \leq 4k$, we have two states: one labeled by $\text{nf}_{srev}(g)$, and one labeled by $\text{nf}_{srev}(g)\$$. We also have one fail state $F$ and one additional state $\lambda'$. The accept states of $M$ are $\lambda'$ and $\lambda\$$, and the initial state is $\lambda$. We define the transition function $\delta$ with three parts:
$$\text{diff}(s, (a, b)) = \begin{cases} \text{nf}_{srev}(a^{-1}gb) & \text{if } s \text{ is labeled by } \text{nf}_{srev}(g),\\
& \hspace{8pt} \text{by nf}_{srev}(g)\$, \text{ or by } \text{nf}_{srev}(g)' \text{ for some } g \in G \\
F & \text{if } s = F \\
\end{cases} $$
$$\text{vterm}(s, (a, b)) = \begin{cases} F & \text{if } s = F \\ 
\$ & \text{if } s \neq F \text{ and } b = \$ \\ 
F & \text{if } s = \text{nf}_{srev}(g)\$ \text{ for some } g \in G \text{ and } b \neq \$ \\ 
\lambda & \text{otherwise} \\
\end{cases} $$
$$\text{srev}(s, (a, b)) = \begin{cases} 1 & \text{if } a, b \in A \text{ and } b \prec a \\
1 & \text{if } s = \lambda' \text{ and } a = b \\
0 & \text{otherwise} \\
\end{cases} $$

Now, we define our transition function by 
$$\delta(s, (a, b)) = \begin{cases}
F & \text{if diff}(s, (a, b)) = F \\
& \hspace{16pt} \text{or vterm}(s, (a, b)) = F\\
& \hspace{16pt} \text{or } \abs{\text{diff}(s, (a, b))} > 4k \\
\lambda' & \text{if diff}(s, (a, b)) = \lambda, \\
& \hspace{16pt} \text{vterm}(s, (a, b)) = \lambda, \\
& \hspace{16pt} \text{and srev}(s, (a, b)) = 1 \\
\text{diff}(s, (a, b))\text{vterm}(s, (a, b)) & \text{otherwise}
\end{cases} $$
The function $\text{diff}(s, (a, b))$ tracks the word difference between the two input words $u$ and $v$ as long as these words $4k$-fellow travel, $u$ does not terminate before $v$, and $v$ does not have a padding letter between letters from $A$. The function $\text{vterm}(s, (a, b))$ tracks whether $v$ has terminated and whether $v$ has a padding letter between letters from $A$. The function $\text{srev}(s, (a, b))$ tracks whether the most recent pair of non-identical letters was in decreasing order.

We now prove that the language accepted by $M$ is $L^p$. We first notice that $\text{vterm}$ prevents $M$ from accepting any words which are not padded pairs, so we can restrict our proof to deal only with padded pairs.

Suppose $(u, v)^p$ is accepted by $M$. Recall that $\widehat{\delta}(s, w)$ is the state of $M$ after starting at a state $s$ and reading a word $w$. Then $\widehat{\delta}(\lambda, (u, v)^p)$ is either $\lambda\$$ or $\lambda'$. In particular, this requires $u^{-1}v =_G 1$, since $\text{diff}$ (which tracks the word difference between prefixes of $u$ and of $v$) ended at $\lambda$. Further, reading $(u, v)^p$ avoids landing at $F$, so $\abs{\text{diff}} \leq 4k$ at each step; thus, $u$ and $v$ must $4k$-fellow travel. We now consider the role of $\text{vterm}$: either this function ended at $\$$, or $\text{vterm}$ ended at $\lambda$ and $\text{srev}$ ended at $1$. In the first case, $(u, v)^p$ reached a padding letter in the second coordinate, meaning that $\abs{v} < \abs{u}$, so $v <_{srev} u$. In the second case, $(u, v)^p$ had no padding symbols, so $\abs{u} = \abs{v}$, but $\text{srev}$ ended at $1$. Then either $\text{last}(v) \prec \text{last}(u)$, so that $v <_{srev} u$, or the state before reading the last letter of $(u, v)^p$ was also $\lambda'$ and $\text{last}(u) = \text{last}(v)$. Continuing this reasoning, we see that $u =_{A^\ast} u_1u_2$, $v =_{A^\ast} v_1u_2$, and $\text{last}(v_1) \prec \text{last}(u_1)$ for some words $u_1, u_2, v_1 \in A^\ast$. In this case, we again have that $v <_{srev} u$. Thus, the language accepted by $M$ is a subset of $L^p$.

Now, suppose that $(u, v)^p$ satisfies $v <_{srev}u$, $u=_G v$, and $u$ and $v$ $4k$-fellow travel. Then starting at $\lambda$ and reading $(u, v)^p$, $\abs{\text{diff}} \leq 4k$ at each step, since $u$ and $v$ $4k$-fellow travel and $\text{vterm}$ is never $F$ because $(u, v)^p$ is a padded pair, so we never reach the state $F$. Since $u =_G v$, $\widehat{\delta}(\lambda, (u, v)^p)$ must be $\lambda$, $\lambda\$$, or $\lambda'$. If $\abs{v} < \abs{u}$, then we have $\text{vterm}$ reaches $\$$ after $v$ ends, so we must end at $\lambda\$$. Otherwise, since $v <_{srev} u$, we must have that $v$ is reverse lexicographically smaller than $u$. Thus, $\text{srev}$ ends at $1$, meaning that we end at $\lambda'$. In either case, $M$ accepts $(u, v)^p$, so $L^p$ is a subset of the language accepted by $M$.  Therefore, $M$ accepts exactly the language $L^p$.

Now, we consider $L$. We first note that $u =_G v$ for all $(u, v) \in L$. Further, for each word $u$ which is not a short reverse-lexicographic normal form, there exists some $v \in A^\ast$ with $v <_{srev} u$ and $u$ and $v$ $2k$-fellow travel by Lemma \ref{doubleFellowTravel}, so $(u, v) \in L$. The ordering $<_{srev}$ is well-founded, so each word which is not a short reverse-lexicographic normal form can be rewritten to its normal form using finitely many rules from $L$, so any two words which are equal in $G$ can be rewritten to each other using finitely many relations of the form $u = v$ with $(u, v) \in L$. Thus, $G = \text{Mon}\langle A | L \rangle$, so $L$ is a prefix-rewriting system for $G$. Because $<_{srev}$ is a well-founded ordering and for all $(u, v) \in L$ we have $v <_{srev} u$, $L$ is terminating, and because each word $u \in A^\ast$ which is not the minimal representative of a group element under this total ordering is the left-hand side of a pair, we have that $L$ has unique normal forms, hence $L$ is convergent. Further, since $L^p$ is the language accepted by $M$, $L$ is synchronously regular. Finally, since $L$ never increases length, we have that $L$ is a length non-increasing synchronously regular convergent prefix-rewriting system for $G$ with generating set $A$. By Theorem \ref{lniToAutost}, $G$ is autostackable.

To create a length non-increasing synchronously regular CP-RS which is also \propL, we consider a sequence of languages. First, let $L' = L \cap (\cali{N}A \times A^\ast)^p$; that is, $L'$ consists of pairs $(u, v)$ such that every proper prefix of $u$ is a short reverse lexicographic normal form. By a similar argument as in the proof of Theorem \ref{lniToAutost}, $L'$ is still a length non-increasing synchronously regular prefix-rewriting system for $G$ with generating set $A$, with normal form set $\cali{N}$. Next, we define $L'_{2} = L' \cap (A^\ast \times A^\ast\$\$)$, $L'_{1} = L' \cap (A^\ast \times A^\ast \$)$, and $L'_0 = L' \cap (A^\ast \times A^\ast)$. That is, $L'_{i}$ is the subset of $L'$ consisting of pairs $(u, v)$ where $\abs{u} = \abs{v} + i$. We next recursively define three languages for each letter in $A$, and two languages $L_i$ for $i = 1, 2$: let 
$$L_{2, a} = \left(L'_2 \cap (A^\ast \times A^\ast a)^p\right) \setminus \left(\bigcup_{b \prec a} \pi_1(L_{2,b})\times A^\ast\right)^p,$$
$$L_2 = \bigcup_{c \in A} \pi_1(L_{2, c}) \times A^\ast,$$
$$L_{1, a} = \left(L'_1 \cap (A^\ast \times A^\ast a)^p\right) \setminus \left(L_2 \cup \bigcup_{b \prec a} \left(\pi_1(L_{1,b})\times A^\ast\right)\right)^p,$$
$$L_1 = \bigcup_{c \in A} \pi_1(L_{1, c}) \times A^\ast,$$
and 
$$L_{0, a} = \left(L'_0 \cap (A^\ast \times A^\ast a)^p \right) \setminus \left(L_2 \cup L_1 \cup \bigcup_{b \prec a} \left(\pi_1(L_{0, b}) \times A^\ast \right)\right)^p.$$
That is, $L_{i, a}$ is the set of all $(u, v)^p \in L'$ such that $\abs{u} = \abs{v} + i$, $\text{last}(v) = a$, and there is no $v' \in A^\ast$ such that $(u, v') \in L'$ and $v' <_{srev}v$. Again, using a similar argument as at the start of the proof of Theorem \ref{lniToAutost}, we have that $L'' = \cup_{i = 0}^2\cup_{a \in A} L_{i, a}$ is a length non-increasing synchronously regular prefix-rewriting system for $G$ with generating set $A$. Moreover, in Lemma \ref{doubleFellowTravel} we showed that any non-normal form geodesic $u$ $2k$-fellow travels a word $v$ with $u =_G v$, $\abs{v} = \abs{u}$ and $\text{last}(v) = \text{last}(\text{nf}_R(u))$, so for any pair $(u, v) \in L''$ where $u$ is a geodesic, we have that $\text{last}(v) = \text{last}(\text{nf}_R(u))$. In the case that $u$ is not geodesic, since $u = wl$ for some geodesic normal form $w$ and some letter $l \in A$, we have that $u$ must have length at most $\abs{\text{nf}_R(u)} + 2$, hence it $k$-fellow travels some shorter word $v_1$, which then $k$-fellow travels some geodesic $v_2$, which then $2k$-fellow travels some geodesic $v$ with $\text{last}(v) = \text{last}(\text{nf}_R(u))$. Thus, each non-geodesic $u$ $4k$-fellow travels a geodesic word $v$ with $\text{last}(v) = \text{last}(\text{nf}_R(u))$, so each rule $(u, v)$ in $L''$ has $\text{last}(v) = \text{last}(\text{nf}_R(u))$.. Thus, $L''$ is a length non-increasing \propL synchronously regular CP-RS, completing the proof of part (a) of the theorem. Now, $L''$ satisfies the hypotheses of Theorem \ref{extraToGeo}, so there is a geodesically autostackable structure for $G$ with generating set $A$.
\end{proof}

When working with the finite state automata constructed through Theorem \ref{lniToAutost}, we notice that there is some potential room for improvement when working with the rewriting systems from pairs $(G, A)$ with FFTP. In particular, rather than using the pumping number for an FSA accepting $L$ (or $L''$), we can use four times the fellow traveler constant. This can decrease the size of the automata created in Theorem \ref{lniToAutost}. There are pairs $(G, A)$ for which Theorem \ref{lniToAutost} is useful which do not have FFTP, so we opted for a proof covering a wider class of groups rather than the more efficient construction in Section \ref{ThmBC}.

\section{Disproving the converse of part (a) of the Main Theorem}\label{noAConverse}
A natural question, given the first part of Theorem \ref{mainThm}, is whether having a length non-increasing regular CP-RS for a pair $(G, A)$ implies that $(G, A)$ has FFTP. In this section, we answer the question in the negative, using the following example.

\setcounter{thm}{4}
\begin{thm} The group $G = \bb{Z}^2 \rtimes \bb{Z}_2 = \langle a, b, t \mid [a, b] = 1, t^2 = 1, tat = b \rangle$ has a length non-increasing regular CP-RS with generating set $A = \set{a, t}$, but the pair $(G, A)$ does not have FFTP. \end{thm}
\begin{proof}
\begin{figure}
\begin{tikzpicture}[scale = 1.5, every node/.style={circle, draw, fill= black, scale = 0.4}]
\foreach \i in {-2, -1, ..., 2}{
	\draw [line width = 2pt, -{Latex[length = 3mm, width = 2mm]}] (\i, -2.5) -- (\i, -1.75);
	\draw [line width = 2pt] (\i, 2.25) -- (\i, 2.5);
	\foreach \j in {-2, -1, ..., 1}{
		\draw [line width = 2pt, -{Latex[length = 3mm, width = 2mm]}] (\i, \j + 0.25) -- (\i, \j + 1.25);
	}
}
\foreach \i in {-2, -1, ..., 2}{
	\draw [preaction = {draw, line width = 3pt, white}, -{Latex[length = 3mm, width = 2mm]}] (-2.5, \i) -- (-2.25, \i);
	\draw [preaction = {draw, line width = 3pt, white}] (1.75, \i) -- (2.5, \i);
	\foreach \j in {-2, -1, ..., 1}{
		\draw [preaction = {draw, line width = 3pt, white}, -{Latex[length = 3mm, width = 2mm]}] (\j - 0.25, \i) -- (\j + 0.75, \i);
	}
}
\draw [line width = 2pt] (-2.5, 0) -- (2.5, 0);
\foreach \i in {-2, -1, ..., 2}{
	\foreach \j in {-2, -1, ..., 2}{
		\draw [line width = 2pt, {Latex[length = 2mm, width = 2mm]}-{Latex[length = 2mm, width = 2mm]}] (\i - 0.25, \j) -- (\i, \j + 0.25);
		\node at (\i - 0.25, \j){};
		\node at (\i, \j + 0.25){};
	}
}
\node[circle, draw, scale = 2] at (-0.25, 0){};
\node (1) at (-0.25, 0)[label = below:\Huge $1_G$]{};

\end{tikzpicture}
\caption{A portion of the Cayley graph $\Gamma_{G, A}$, with spanning tree of normal forms drawn in bold. All horizontal and vertical arrows are labeled by $a$ and point right or up, while all diagonal arrows are labeled by $t$ and are bidirectional.}
\label{fig: semidirect}
\end{figure}
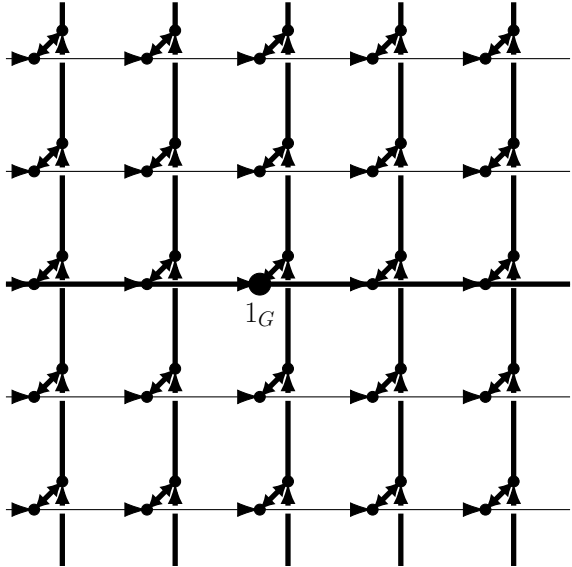
Consider $G = \bb{Z}^2 \rtimes \bb{Z}_2$ and $A = \set{a, t}$. Elder uses this example and proves that the pair $(G, A)$ does not have FFTP in \cite{Elder2005}. Consider the shortlex normal form set $\cali{N}$ with $a < a^{-1} < t < t^{-1}$ as our ordering on $A^{\pm 1}$. An illustration of part of the Cayley graph, with normal forms indicated, is given as Figure \ref{fig: semidirect}. We see that $\cali{N}$ consists of $\set{a^ita^jt \mid i, j \in \bb{Z}, j \neq 0}$ and all prefixes of words in this language. Notably, $\cali{N}$ is regular. Further, whenever $u, v \in \cali{N}$ with $d(u, v) = 1$, we have that $u$ and $v$ 4-fellow travel. The only non-trivial case to check for 4-fellow traveling is when $u = a^ita^jt$ and $v = a^{i + 1}ta^jt$, with $v =_G ua$. In this case, if $i$ is non-negative, $u$ and $v$ follow a common path of length $i$, then have a word difference of $ta$ for a single step, then have a word difference of $tata$ for $j$ letters, and finally have a word difference of $a$ at the final pair of vertices. If $i$ is negative, we have a similar scenario. All other options for $u$ and $v$ have either $u$ as a prefix of $v$ or $v$ as a prefix of $u$, so $u$ and $v$ 1-fellow travel. Thus, we have a regular language of shortlex normal forms, with any pair of normal forms that differ by a single edge $4$-fellow traveling, hence a shortlex automatic structure \cite[Theorem 2.3.5]{ECHLPT1992}. Every shortlex automatic structure is a length non-increasing regular CP-RS (this follows from the proof of \cite[Lemma 5.1]{Otto1999}), so this example has a length non-increasing regular CP-RS but not FFTP. The set of rules $R$ for this CP-RS is given below, for sake of completeness:
\begin{align*}
a^ia^{-\text{sgn}(i)} &\to a^{i - \text{sgn}(i)} && \text{ for all } i \in \bb{Z} \setminus \set{0}\\
a^ita^ja^{-\text{sgn}(j)} &\to a^ita^{j - \text{sgn}(j)} && \text{ for all } i \in \bb{Z} \text{ and } j \in \bb{Z} \setminus \set{0}\\
a^itt &\to a^i && \text{ for all } i \in \bb{Z} \\
a^ita^jt^2 &\to a^ita^j && \text{ for all } i \in \bb{Z} \text{ and } j \in \bb{Z} \setminus \set{0} \\
a^ita^jta^\epsilon &\to a^{i + \epsilon}ta^jt && \text{ for all } i \in \bb{Z} \text{, } j \in \bb{Z} \setminus \set{0} \text{, and } \epsilon \in \set{1, -1} \\
\end{align*}
It is worth noting that this CP-RS is \propL, hence $G$ is geodesically autostackable with the given generating set.
\end{proof}

\section*{Acknowledgements}
The author received partial support from Simons Foundation Collaboration Grant number 581433.

\end{document}